\documentclass{amsart}
\usepackage[margin=1.4in]{geometry}

\usepackage{amsfonts}

\usepackage{setspace}
\usepackage{graphicx}
\usepackage{eufrak}
\usepackage{mathrsfs}
\usepackage{yfonts}
\usepackage{hyperref}
\usepackage{graphicx}
\usepackage{amsmath, amsthm, amscd, amsfonts, amssymb, graphicx, color}
\usepackage{url}
\usepackage{enumerate}
\makeatletter
\let\reftagform@=\tagform@
\def\tagform@#1{\maketag@@@{(\ignorespaces\textcolor{blue}{#1}\unskip\@@italiccorr)}}
\renewcommand{\eqref}[1]{\textup{\reftagform@{\ref{#1}}}}
\makeatother
\usepackage{hyperref}
\hypersetup{colorlinks=true, linkcolor=red, anchorcolor=green,
    citecolor=cyan, urlcolor=red, filecolor=magenta, pdftoolbar=true}

\usepackage{epsfig}        
\usepackage{epic,eepic}       

\newtheorem{theorem}{Theorem}
\theoremstyle{plain}

\newtheorem{corollary}{Corollary}

\newtheorem{lemma}{Lemma}

\newtheorem{remark}{Remark}

\numberwithin{equation}{section}

 \DeclareMathOperator{\spe}{sp}

\def\etal{et al.\,}

\begin{document}
  \title[Numerical radius inequalities  for Hilbert Space Operators ]
{ Numerical radius inequalities   for Hilbert Space Operators }
\author[M.W. Alomari]{Mohammad .W. Alomari}

\address{Department of Mathematics, Faculty of Science and Information
	Technology, Irbid National University, 2600 Irbid 21110, Jordan.}
\email{mwomath@gmail.com}

\date{\today}
\subjclass[2010]{Primary: 47A12, 47A30   Secondary: 15A60, 47A63.
}

\keywords{Numerical radius, Operator norm, mixed Schwarz
	inequality, H\"{o}lder-McCarty inequality}
\begin{abstract}
	In this work, an improvement of H\"{o}lder-McCarty inequality is
	established. Based on that, several refinements of the generalized
	mixed Schwarz inequality are obtained. Consequently,  some new
	numerical radius inequalities are  proved. New
	inequalities for numerical radius of $n\times n$ matrix of Hilbert
	space operators are proved as well. Some refinements of some earlier
	results were proved in literature are also given. Some of the presented
	results are refined and it shown to be better than earlier results
	were proved in literature.
\end{abstract}

\maketitle
\section{Introduction}

Let $\mathscr{B}\left( \mathscr{H}\right) $ be the Banach algebra
of all bounded linear operators defined on a complex Hilbert space
$\left( \mathscr{H};\left\langle \cdot ,\cdot \right\rangle
\right)$  with the identity operator  $1_\mathscr{H}$ in
$\mathscr{B}\left( \mathscr{H}\right) $. A bounded linear
operator $A$ defined on $\mathscr{H}$ is selfadjoint if and only
if $ \left\langle {Ax,x} \right\rangle \in \mathbb{R}$ for all
$x\in \mathscr{H}$. The spectrum of an operator $A$ is the set of all $\lambda \in \mathbb{C}$  for which the operator $\lambda I - A$ does not have a bounded linear operator inverse, and is denoted by
$\spe\left(A\right)$. Consider the real vector space
$\mathscr{B}\left( \mathscr{H}\right)_{sa}$ of self-adjoint
operators on $ \mathscr{H}$ and its positive cone
$\mathscr{B}\left( \mathscr{H}\right)^{+}$ of positive operators
on $\mathscr{H}$. Also, $\mathscr{B}\left(
\mathscr{H}\right)_{sa}^I$ denotes the convex set of bounded
self-adjoint operators on the Hilbert space $\mathscr{H}$ with
spectra in a real interval $I$. A
partial order is naturally equipped on $\mathscr{B}\left(
\mathscr{H}\right)_{sa}$ by defining $A\le B$ if and only if
$B-A\in   \mathscr{B}\left( \mathscr{H}\right)^{+}$.  We write $A
> 0$ to mean that $A$ is a strictly positive operator, or
equivalently, $A \ge 0$ and $A$ is invertible. When $\mathscr{H} =
\mathbb{C}^n$, we identify $\mathscr{B}\left( \mathscr{H}\right)$
with the algebra $\mathfrak{M}_{n\times n}$ of $n$-by-$n$ complex
matrices. Then, $\mathfrak{M}^{+}_{n\times n}$ is just the cone of
$n$-by-$n$ positive semidefinite matrices.

For a bounded linear operator $T$ on a Hilbert space
$\mathscr{H}$, the numerical range $W\left(T\right)$ is the image
of the unit sphere of $\mathscr{H}$ under the quadratic form $x\to
\left\langle {Tx,x} \right\rangle$ associated with the operator.
More precisely,
\begin{align*}
W\left( T \right) = \left\{ {\left\langle {Tx,x} \right\rangle :x
	\in \mathscr{H},\left\| x \right\| = 1} \right\}
\end{align*}
Also, the numerical radius is defined to be
\begin{align*}
w\left( T \right) = \sup \left\{ {\left| \lambda\right|:\lambda
	\in W\left( T \right) } \right\} = \mathop {\sup }\limits_{\left\|
	x \right\| = 1} \left| {\left\langle {Tx,x} \right\rangle }
\right|.
\end{align*}

The spectral radius of an operator $T$ is defined to be
\begin{align*}
r\left( T \right) = \sup \left\{ {\left| \lambda\right|:\lambda
	\in \spe\left( T \right) } \right\}
\end{align*}

We recall that,  the usual operator norm of an operator $T$ is
defined to be
\begin{align*}
\left\| T \right\| = \sup \left\{ {\left\| {Tx} \right\|:x \in
	H,\left\| x \right\| = 1} \right\}.
\end{align*}
and
\begin{align*}
\ell \left( T \right): &= \inf \left\{ {\left\| {Tx} \right\|:x
	\in \mathscr{H},\left\| x \right\| = 1} \right\}
\\
&=      \inf \left\{ {\left|\left\langle {Tx,y} \right\rangle
	\right|:x,y \in
	\mathscr{H},\left\| x \right\| =\left\| y \right\|= 1} \right\}.
\end{align*}

 It's well known that   the numerical radius is not submultiplicative, but it is satisfies  $w(TS)\le 4w\left(T\right) \left(S\right)$ for all $T,S\in \mathscr{B}\left(  \mathscr{H}\right)$. In particular if $T,S$ are commute,  then $w(TS)\le 2w\left(T\right) \left(S\right)$. Moreover, if $T,S$ are normal  then $w\left(\cdot\right)$ is submultiplicative $w(TS)\le w\left(T\right) \left(S\right)$. Denote  $|T|=\left(T^*T\right)^{1/2}$  the absolute value of the operator $T$. Then we have $w\left(|T|\right) = \|T\|$. It's convenient to mention  that, the numerical radius norm is weakly unitarily invariant; i.e., $w\left(U^*TU\right) = w\left(T\right)$  for all unitary $U$. Also, let us  not miss the chance to mention the important property that  $w\left(T\right) = w\left(T^*\right)$  and $w\left(T^*T\right) = w\left(TT^*\right)$ for every $T\in \mathscr{B}\left( \mathscr{H}\right)$.

A popular problem is the following: does the numerical radius of the product of operators commute, i.e., $w(TS)= w\left(ST\right)$  for any operators $T,S\in \mathscr{B}\left(\mathscr{H}\right)$?

This problem has been given serious attention by many authors and in several resources (see \cite{G}, for example). Fortunately, it has been shown recently that,  for one of such  operators must be a multiple of a unitary operator, and we need  only to check $w\left(TS\right)=w\left(ST\right)$ for all rank one operators $S\in \mathscr{B}\left( \mathscr{H}\right)$ to arrive at the conclusion. This fact was proved by Chien \etal  in \cite{CGLTW}. For other related problems involving  numerical ranges and radius see \cite{CGLTW} and \cite{CKN}
as well as the elegant work of Li \cite{LTWW} and the  references
therein. For more classical and recent properties of  numerical range  and radius, see \cite{CGLTW} \cite{CKN}, \cite{LTWW} and the comprehensive books \cite{B},  \cite{H1} and \cite{H2}.

On the other hand, it is well known that $w\left(\cdot\right)$
defines an operator norm on $\mathscr{B}\left( \mathscr{H}\right)
$ which is equivalent to operator norm $\|\cdot\|$. Moreover, we
have
\begin{align}
\frac{1}{2}\|T\|\le w\left(T\right) \le \|T\|\label{eq1.1}
\end{align}
for any $T\in \mathscr{B}\left( \mathscr{H}\right)$. The
inequality is sharp.

In 2003, Kittaneh \cite{FK1} refined the right-hand side of
\eqref{eq1.1}, where he proved that
\begin{align}
w\left(T\right) \le
\frac{1}{2}\left(\|T\|+\|T^2\|^{1/2}\right)\label{eq1.2}
\end{align}
for any  $T\in \mathscr{B}\left( \mathscr{H}\right)$.

After that in 2005, the same author in \cite{FK} proved that
\begin{align}
\frac{1}{4}\|A^*A+AA^*\|\le  w^2\left(T\right) \le
\frac{1}{2}\|A^*A+AA^*\|.\label{eq1.3}
\end{align}
The inequality is sharp. This inequality was also reformulated and generalized in \cite{EF} but in terms of Cartesian
decomposition.

In 2007, Yamazaki \cite{Y} improved both \eqref{eq1.1} and
\eqref{eq1.2} by proving that
\begin{align}
w\left( T \right) \le \frac{1}{2}\left( {\left\| T \right\| +
	w\left( {\widetilde{T}} \right)} \right) \le \frac{1}{2}\left(
{\left\| T \right\| + \left\| {T^2 } \right\|^{1/2} }
\right)\label{eq1.4}
\end{align}
where $\widetilde{T}=|T|^{1/2}U|T|^{1/2}$ with unitary $U$.

In 2008, Dragomir \cite{D4} used Buzano inequality to improve
\eqref{eq1.1}, as follows:
\begin{align}
w^2\left( T \right) \le \frac{1}{2}\left( {\left\| T \right\| +
	w\left( {T^2} \right)} \right) \label{eq1.5}
\end{align}
This result was also recently generalized by Sattari \etal in
\cite{SMY}.

This work, is divided into three sections, after this
introduction, Section \ref{sec2} is devoted to recall some facts
about superquadratic functions and the mixed Schwarz inequality.
In Section \ref{sec3}, we  refine the  Jesnen and
H\"{o}lder--McCarty  inequalities for positive operators which in
turn   allow us to refine the generalized mixed Schwarz inequality
with of its some consequences.   In Section \ref{sec4}, new
inequalities for numerical radius of $n\times n$ matrix of Hilbert
space operators are proved. Some refinements of some earlier
results were proved in literature are also given.

\section{Lemmas }\label{sec2}

\subsection{Superquadratic functions}

A function $f:J\to \mathbb{R}$ is called convex iff
\begin{align*}
f\left( {t\alpha +\left(1-t\right)\beta} \right)\le tf\left(
{\alpha} \right)+ \left(1-t\right) f\left( {\beta} \right),
\end{align*}
for all points $\alpha,\beta \in J$ and all $t\in [0,1]$. If $-f$
is convex then we say that $f$ is concave. Moreover, if $f$ is
both convex and concave, then $f$ is said to be affine.

Geometrically, for two point $\left(x,f\left(x\right)\right)$ and
$\left(y,f\left(y\right)\right)$  on the graph of $f$ are on or
below the chord joining the endpoints  for all $x,y \in I$, $x <
y$. In symbols, we write
\begin{align*}
f\left(t\right)\le   \frac{f\left( y \right)  - f\left( x \right)
}{y-x}   \left( {t-x} \right)+ f\left( x \right)
\end{align*}
for any $x \le t \le y$ and $x,y\in J$.

Equivalently, given a function $f : J\to \mathbb{R}$, we say that
$f$ admits a support line at $x \in J $ if there exists a $\lambda
\in \mathbb{R}$ such that
\begin{align*}
f\left( t \right) \ge f\left( x \right) + \lambda \left( {t - x}
\right)
\end{align*}
for all $t\in J$.

The set of all such $\lambda$ is called the subdifferential of $f$
at $x$, and it's denoted by $\partial f$. Indeed, the
subdifferential gives us the slopes of the supporting lines for
the graph of $f$. So that if $f$ is convex then $\partial f(x) \ne
\emptyset$ at all interior points of its domain.

From this point of view  Abramovich \etal \cite{SJS} extend the
above idea for what they called superquadratic functions. Namely,
a function $f:[0,\infty)\to \mathbb{R}$ is called superquadratic
provided that for all $x\ge0$ there exists a constant $C_x\in
\mathbb{R}$ such that
\begin{align*}
f\left( t \right) \ge f\left( x \right) + C_x \left( {t - x}
\right) + f\left( {\left| {t - x} \right|} \right)
\end{align*}
for all $t\ge0$. We say that $f$ is subquadratic if $-f$ is
superquadratic. Thus, for a superquadratic function we require
that $f$ lie above its tangent line plus a translation of $f$
itself.

Prima facie, superquadratic function  looks  to be stronger than
convex function itself but if $f$ takes negative values then it
may be considered as a weaker function. Therefore, if $f$ is
superquadratic and non-negative. Then $f$ is convex and increasing
\cite{SJS}.

Moreover,   the following result holds for superquadratic
function.

\begin{lemma}\cite{SJS}
	\label{lemma1}Let $f$ be superquadratic function. Then
	\begin{enumerate}
		\item $f\left(0\right)\le 0$
		
		\item If $f$ is differentiable and $f(0)=f^{\prime}(0)=0$, then $C_x=f^{\prime}(x)$  for all $x\ge0$.
		
		\item If $f(x)\ge0$ for all $x\ge 0$, then $f$ is convex and $f(0)=f^{\prime}(0)=0$.
	\end{enumerate}
\end{lemma}

The next result gives a sufficient condition when convexity
(concavity) implies super(sub)quaradicity.

\begin{lemma}\cite{SJS}
	\label{lemma2}If $f^{\prime}$ is convex (concave) and
	$f(0)=f^{\prime}(0)=0$, then is super(sub)quadratic. The converse
	of is not true.
\end{lemma}

\begin{remark}
	Subquadraticity does always not imply concavity; i.e.,  there exists a subquadratic function which is convex. For example, $f(x)=x^p$, $x\ge 0$ and $1\le  p \le2$ is subquadratic and convex. For more about subquadratic see \cite{KLPP}.
\end{remark}

Among others,  Abramovich \etal \cite{SJS} proved that the
inequality
\begin{align}
f\left( {\int {\varphi d\mu } } \right) \le   \int {f\left( {\varphi \left( s \right)} \right)-f\left( {\left| {\varphi \left( s \right) - \int {\varphi d\mu } } \right|} \right)d\mu \left( s \right)}\label{eq2.1}
\end{align}
holds for all probability measures $\mu$ and all nonnegative,
$\mu$-integrable functions $\varphi$ if and only if $f$ is superquadratic. This inequality plays a main role overall our presented results below.\\

\subsection{The mixed Schwarz inequality}

The mixed Schwarz inequality was introduced in \cite{TK}, as
follows:
\begin{lemma}
	\label{lemma3}  Let  $A\in \mathscr{B}\left( \mathscr{H}\right)^+ $, then
	\begin{align}
	\left| {\left\langle {Ax,y} \right\rangle} \right|  ^2  \le \left\langle {\left| A \right|^{2\alpha } x,x} \right\rangle \left\langle {\left| {A^* } \right|^{2\left( {1 - \alpha } \right)} y,y} \right\rangle, \qquad 0\le \alpha \le 1. \label{eq2.2}
	\end{align}
	for any   vectors $x,y\in \mathscr{H}$
\end{lemma}

In order to generalize \eqref{eq2.2}, Kittaneh in \cite{FK4}  used
the key lemma
\begin{lemma}
	\label{lemma4}Let  $A,B\in \mathscr{B}\left( \mathscr{H}\right)^+$. Then $\left[ {\begin{array}{*{20}c} A & {C^* }  \\C & B  \\\end{array}} \right]$  is positive in $\mathscr{B}\left(\mathscr{H}\oplus \mathscr{H}\right)$ if and only if $
	\left| {\left\langle {Cx,y} \right\rangle } \right|^2  \le \left\langle {Ax,x} \right\rangle \left\langle {By,y} \right\rangle $ for every vectors $x,y\in \mathscr{H}$,
\end{lemma}
to prove  that
\begin{lemma}
	\label{lemma5}  Let  $A,B\in \mathscr{B}\left( \mathscr{H}\right)$ such that $|A|B=B^*|A|$.
	If $f$ and $g$ are nonnegative continuous functions on $\left[0,\infty\right)$ satisfying $f(t)g(t) =t$ $(t\ge0)$, then
	\begin{align}
	\left| {\left\langle {ABx,y} \right\rangle } \right| \le r\left(B\right)\left\| {f\left( {\left| A \right|} \right)x} \right\|\left\| {g\left( {\left| {A^* } \right|} \right)y} \right\|\label{kittaneh.ineq}
	\end{align}
	for any   vectors $x,y\in  \mathscr{H} $.
\end{lemma}
Clearly, by setting  $B=1_{\mathscr{H}}$ and choosing
$f(t)=t^{\alpha}$, $g(t)=t^{1-\alpha}$ we refer to \eqref{eq2.2}.

The following interesting estimates of spectral radius also
obtained by Kittaneh in \cite{F}.
\begin{lemma}
	\label{lemma6}If $A,B\in \mathscr{B}\left( \mathscr{H}\right)$.
	Then
	\begin{multline}
	r\left( {AB} \right) \\\le \frac{1}{4}\left( {\left\| {AB}
		\right\| + \left\| {BA} \right\| + \sqrt {\left(\left\| {AB}
			\right\|- \left\| {BA} \right\|\right)^2 + 4\min \left\{ {\left\|
				A \right\|\left\| {BAB} \right\|,\left\| B \right\|\left\| {ABA}
				\right\|} \right\}}} \right) \label{fact3}
	\end{multline}
\end{lemma}

In some of our results we need the following two fundamental norm
estimates, which  are:
\begin{align}
\label{fact1}\left\| {A+ B } \right\|\le \frac{1}{2}\left(
{\left\| A \right\| + \left\| B \right\| + \sqrt
	{\left( {\left\| A \right\| - \left\| B \right\|} \right)^2  +
		4\left\| {A^{1/2} B^{1/2} } \right\|^2 } } \right),
\end{align}
and
\begin{align}
\label{fact2}\left\| {A^{1/2} B^{1/2} } \right\|  \le\left\| {A  B
} \right\| ^{1/2}.
\end{align}
Both estimates are valid for all positive operators $A,B\in
\mathscr{B}\left( \mathscr{H}\right)$.
\section{Refining H\"{o}lder-McCarty inequality and  mixed Schwarz inequality}\label{sec3}
In the this part we give some new refinements of the `mixed'
Schwarz inequality and its generalization based on a new
refinement of H\"{o}lder--McCarty inequality. The next lemma plays
a main role in our main results.
\begin{lemma}
	\label{lemma7}Let $A\in \mathscr{B}\left( \mathscr{H}\right)^+$, then
	\begin{align}
	\label{eq3.1}  \left\langle {Ax,x} \right\rangle^p \le
	\left\langle { A^p x,x} \right\rangle- \left\langle{\left| {A -
			\left\langle { A x,x} \right\rangle 1_{\mathcal{H}} }
		\right|^px,x} \right\rangle \le \left\langle { A^p x,x}
	\right\rangle
	\end{align}
	for all $p \ge2$, and
	\begin{align}
	\label{eq3.2}\left\langle {Ax,x} \right\rangle^p  \ge \left\langle
	{ A^p x,x} \right\rangle- \left\langle{\left| {A - \left\langle {A
				x,x} \right\rangle 1_{\mathcal{H}} } \right|^px,x} \right\rangle
	\end{align}
	for all $0<p<2$ and every $x\in \mathcal{H}$.
\end{lemma}
\begin{proof}
	Since $A$ is positive then there $B\in \mathscr{B}\left(
	\mathscr{H}\right)$ such that $A=B^*B$. Also, since $B^*B$ is
	always positive and selfadjoint, thus by the spectral representation
	theorem  $A$ can be represented as $A=\int_0^\infty  {tdE\left( t
		\right)} $. Employing the inequality \eqref{eq2.1} for the
	superquadratic function  $f\left(t\right)=t^p$, $t\in
	\left[0,\infty\right)$ $p\ge2$, then we have
	\begin{align*}
	\left\langle {Ax,x} \right\rangle^p &= \left( \int_0^\infty  {t\left\langle {dE\left( t \right)x,x} \right\rangle}\right)^p\\
	&\le \int_0^\infty  {t^p\left\langle {dE\left( t \right)x,x} \right\rangle} -\int_0^\infty { \left|t  - \int_0^\infty  {s\left\langle {dE\left( s \right)x,x} \right\rangle}\right|^p \left\langle {dE\left( t \right)x,x} \right\rangle}\\
	&=\left\langle { A^p x,x} \right\rangle- \left\langle{\left| {A -
			\left\langle{A x,x} \right\rangle 1_{\mathcal{H}} } \right|^px,x}
	\right\rangle.
	\end{align*}
	The inequality \eqref{eq3.2} follows in similar manner by applying
	the reverse of  \eqref{eq3.1} for the subquadratic function
	$f(t)=t^p$, $0<p\le 2$.
\end{proof}
The inequalities \eqref{eq3.1} and \eqref{eq3.2}  were proved in
\cite{AM} in different context and only for positive selfadjoint
operators. Also, we should note that, a stronger version for
positive selfadjoint operators was proved earlier in \cite{KS}
(see also \cite{K}) where different approach were used. Our
presented proof above  is more general and completely different.

\begin{remark}
	Let $A\in \mathscr{B}\left( \mathscr{H}\right)^+$, then the
	McCatry inequality reads that
	\begin{align}
	\left\langle {Ax,x} \right\rangle^p  \ge \left\langle { A^p x,x}
	\right\rangle,  0< p \le 1.\label{eq.mc}
	\end{align}
	Using \eqref{eq3.2}, we have the following refinement
	\begin{align*}
	\left\langle {Ax,x} \right\rangle^p  \ge \left\langle { A^p x,x}
	\right\rangle\ge \left\langle { A^p x,x} \right\rangle-
	\left\langle{\left| {A - \left\langle{A x,x} \right\rangle
			1_{\mathcal{H}} } \right|^px,x} \right\rangle, \qquad 0< p \le 1
	\end{align*}
	for every $x\in \mathcal{H}$.
\end{remark}

The following refinement of Cauchy-Schwarz inequality holds.
\begin{lemma}
	Let $A\in \mathscr{B}\left( \mathscr{H}\right)^+$, then
	\begin{align}
	\label{eq3.3}\left|\left\langle {Ax,y} \right\rangle\right|
	^{2p}&\le   \left[\left\langle { A^p x,x} \right\rangle-
	\left\langle{\left| {A - \left\langle { A x,x} \right\rangle
			1_{\mathcal{H}} } \right|^px,x} \right\rangle
	\right]\\&\qquad\times \left[\left\langle { A^p y,y}
	\right\rangle- \left\langle{\left| {A - \left\langle
			{ A y,y} \right\rangle 1_{\mathcal{H}} } \right|^py,y} \right\rangle \right]\nonumber \\
	&\le\left\langle { A^p x,x} \right\rangle\left\langle { A^p y,y}
	\right\rangle \nonumber
	\end{align}
	for all $p\ge2$ and every $x,y\in \mathcal{H}$.
\end{lemma}

\begin{proof}
	By Cauchy-Schwarz inequality we have
	\begin{align*}
	\left|\left\langle {Ax,y} \right\rangle\right|^{2} \le
	\left\langle {Ax,x} \right\rangle \left\langle {Ay,y}
	\right\rangle
	\end{align*}
	for every $x,y\in \mathcal{H}$, and this implies that
	\begin{align*}
	\left|\left\langle {Ax,y} \right\rangle\right|^{2p} \le
	\left\langle {Ax,x} \right\rangle^p\left\langle {Ay,y}
	\right\rangle^p,\qquad p\ge2.
	\end{align*}
	Employing \eqref{eq3.1} we get the desired result.
\end{proof}

\begin{corollary}
	If $T\in \mathscr{B}\left( \mathscr{H}\right)$, then
	\begin{align}
	\label{eq3.4}\left| {\left\langle {Tx,y} \right\rangle }
	\right|^{2p}  &\le  \left[ \left\langle { |T|^p x,x}
	\right\rangle- \left\langle{\left| {|T| - \left\langle { |T| x,x}
			\right\rangle 1_{\mathcal{H}} } \right|^px,x} \right\rangle\right]
	\\ &\qquad \times\left[ \left\langle { |T^*|^py,y} \right\rangle-
	\left\langle{\left| {|T^*| - \left\langle{ |T^*|y,y} \right\rangle
			1_{\mathcal{H}} } \right|^py,y} \right\rangle\right] \nonumber
	\\
	&\le  \left\langle { |T|^p x,x} \right\rangle  \left\langle {
		|T^*|^py,y} \right\rangle\nonumber
	\end{align}
	for all $p\ge2$. In particular, we have
	\begin{align}
	\left| {\left\langle {Tx,x} \right\rangle } \right|\le  \left[
	\left\langle { |T|^p x,x} \right\rangle- \left\langle{\left| {|T|
			- \left\langle
			{ |T| x,x} \right\rangle 1_{\mathcal{H}} } \right|^px,x} \right\rangle\right]^{1/p}\label{eq3.5}
	\end{align}
\end{corollary}
\begin{proof}
	Recall that, if $T\in \mathscr{B}\left( \mathscr{H}\right)$, then
	$\left[ {\begin{array}{*{20}c} |T| & {T^* }  \\T & |T^*|
		\\\end{array}} \right]$  is positive in
	$\mathscr{B}\left(\mathscr{H}\oplus \mathscr{H}\right)$, (see
	\cite{FK4}). Therefore, by \eqref{eq2.1} we have
	\begin{align*}
	\left| {\left\langle {Tx,y} \right\rangle } \right|^{2}  \le
	\left\langle {|T|x,x} \right\rangle \left\langle {|T^*|y,y}
	\right\rangle
	\end{align*}
	and this gives by \eqref{eq3.1} that;
	\begin{align*}
	\left| {\left\langle {Tx,y} \right\rangle } \right|^{2p}  &\le  \left\langle {|T|x,x} \right\rangle^{p} \left\langle {|T^*|y,y} \right\rangle^{p}  \\
	&\le \left[ \left\langle { |T|^p x,x} \right\rangle-
	\left\langle{\left| {|T| - \left\langle { |T| x,x} \right\rangle
			1_{\mathcal{H}} } \right|^px,x} \right\rangle\right]
	\\&\qquad\times\left[ \left\langle { |T^*|^py,y} \right\rangle-
	\left\langle{\left| {|T^*| - \left\langle{ |T^*|y,y} \right\rangle
			1_{\mathcal{H}} } \right|^py,y} \right\rangle\right]
	\\
	&\le  \left\langle { |T|^p x,x} \right\rangle  \left\langle {
		|T^*|^py,y} \right\rangle\nonumber
	\end{align*}
	as desired.
\end{proof}

A generalization of the above result in Kittaneh like inequality
\eqref{kittaneh.ineq} is considered in the following result.
\begin{corollary}
	Let  $T,S\in \mathscr{B}\left( \mathscr{H}\right)$ such that
	$|T|S=S^*|T|$. If $f$ and $g$ are nonnegative continuous functions
	on $\left[0,\infty\right)$ satisfying $f(t)g(t) =t$ $(t\ge0)$,
	then
	\begin{align}
	\label{eq3.6}&\left| {\left\langle {TSx,y} \right\rangle } \right|
	\\&\le r\left( S \right) \sqrt[{2p}]{\left\langle { {f^{2p} \left(
				{\left| T \right|} \right)} x,x} \right\rangle-
		\left\langle{\left| {{f^2 \left( {\left| T \right|} \right)} -
				\left\langle
				{ {f^2 \left( {\left| T \right|} \right)} x,x} \right\rangle 1_{\mathcal{H}} } \right|^px,x} \right\rangle}\nonumber\\
	&\qquad\times \sqrt[{2p}]{\left\langle { {g^{2p} \left( {\left|
					T^* \right|} \right)} y,y} \right\rangle- \left\langle{\left|
			{{g^2 \left( {\left| T^* \right|} \right)} - \left\langle
				{ {g^2 \left( {\left| T^* \right|} \right)} y,y} \right\rangle 1_{\mathcal{H}} } \right|^py,y} \right\rangle}\nonumber\\
	&\le r\left( S \right) \sqrt[{2p}]{\left\langle { {f^{2p} \left( {\left| T \right|} \right)} x,x} \right\rangle }
	\sqrt[{2p}]{\left\langle { {g^{2p} \left( {\left| T^* \right|}
				\right)} y,y} \right\rangle }\nonumber
	\end{align}
	for all $p\ge2$ and any   vectors $x,y\in  \mathscr{H} $,
\end{corollary}
\begin{proof}
	Using \eqref{kittaneh.ineq}, and by employing \eqref{eq3.3} we
	have
	\begin{align*}
	\left| {\left\langle {TSx,y} \right\rangle } \right| & \le r\left(S\right)\left\| {f\left( {\left| T \right|} \right)x} \right\|\left\| {g\left( {\left| {T^* } \right|} \right)y} \right\|\\
	&=r\left( S \right)\left\langle {f^2 \left( {\left| T \right|} \right)x,x} \right\rangle ^{1/2} \left\langle {g^2 \left( {\left| {T^* } \right|} \right)y,y} \right\rangle ^{1/2} \\
	&\le r\left( S \right) \sqrt[{2p}]{\left\langle { {f^{2p} \left(
				{\left| T \right|} \right)} x,x} \right\rangle-
		\left\langle{\left| {{f^2 \left( {\left| T \right|} \right)} -
				\left\langle
				{ {f^2 \left( {\left| T \right|} \right)} x,x} \right\rangle 1_{\mathcal{H}} } \right|^px,x} \right\rangle}\\
	&\qquad\times \sqrt[{2p}]{\left\langle { {g^{2p} \left( {\left|
					T^* \right|} \right)} y,y} \right\rangle- \left\langle{\left|
			{{g^2 \left( {\left| T^* \right|} \right)} - \left\langle { {g^2
						\left( {\left| T^* \right|} \right)} y,y} \right\rangle
				1_{\mathcal{H}} } \right|^py,y} \right\rangle}\\
	&\le r\left( S \right) \sqrt[{2p}]{\left\langle { {f^{2p} \left( {\left| T \right|} \right)} x,x} \right\rangle }
	\sqrt[{2p}]{\left\langle { {g^{2p} \left( {\left| T^* \right|}
				\right)} y,y} \right\rangle }\nonumber
	\end{align*}
	which proves the result.
\end{proof}

\begin{corollary}
	Let  $T,S\in \mathscr{B}\left( \mathscr{H}\right)$ such that
	$|T|S=S^*|T|$. Then
	\begin{align}
	\label{eq3.7}  &\left| {\left\langle {TSx,y} \right\rangle } \right|\\ &\le r\left( S \right) \sqrt[{2p}]{\left\langle { {   \left| T \right|^{2p\alpha } } x,x} \right\rangle- \left\langle{\left| {{  \left| T \right|^{2\alpha} } - \left\langle    { {\left| T \right|^{2\alpha}} x,x} \right\rangle 1_{\mathcal{H}} } \right|^{p}x,x} \right\rangle}\nonumber\\
	&\qquad\times \sqrt[{2p}]{\left\langle { {   \left| T^* \right|^{2p\left(1-\alpha\right) } } x,x} \right\rangle- \left\langle{\left| {{  \left| T^* \right|^{2\left(1-\alpha\right)} } - \left\langle   { {\left| T^* \right|^{2\left(1-\alpha\right)}} x,x} \right\rangle 1_{\mathcal{H}} } \right|^{p}x,x} \right\rangle}\nonumber
	\\
	&\le r\left( S \right) \sqrt[{2p}]{\left\langle { {   \left| T \right|^{2p\alpha } } x,x} \right\rangle }
	\sqrt[{2p}]{\left\langle { {   \left| T^* \right|^{2p\left(1-\alpha\right) } } x,x} \right\rangle }\nonumber
	\end{align}
	for all $p\ge2$  and any   vectors $x,y\in  \mathscr{H} $. In particular, we have
	\begin{align}
	\label{eq3.8}&\left| {\left\langle {TSx,y} \right\rangle } \right| \nonumber\\&\le r\left( S \right) \sqrt[{4}]{\left\langle { {   \left| T \right|^{2 } } x,x} \right\rangle- \left\langle{\left| {{  \left| T \right| } - \left\langle   { {\left| T \right|} x,x} \right\rangle 1_{\mathcal{H}} } \right|^{2}x,x} \right\rangle}\\
	&\qquad\times \sqrt[{4}]{\left\langle { {   \left| T^* \right|^{2  } } x,x} \right\rangle- \left\langle{\left| {{  \left| T^* \right|  } - \left\langle { {\left| T^* \right| } x,x} \right\rangle 1_{\mathcal{H}} } \right|^{2}x,x} \right\rangle}\nonumber
	\\
	&\le r\left( S \right) \sqrt[{4}]{\left\langle { {   \left| T \right|^{2 } } x,x} \right\rangle}\sqrt[{4}]{\left\langle { {   \left| T^* \right|^{2 } } x,x} \right\rangle}.\nonumber
	\end{align}
\end{corollary}
\begin{proof}
	Setting $f\left(t\right)=t^{\alpha}$ and
	$g\left(t\right)=t^{1-\alpha}$ in \eqref{eq3.6} we get the
	inequality \eqref{eq3.7}. Choosing $p=2$ and $\alpha=\frac{1}{2}$
	in \eqref{eq3.7}, we get the second inequality \eqref{eq3.8}.
\end{proof}

\section{Numerical radius inequalities}\label{sec4}

This section is divided into two parts; the first part concerning
numerical inequalities for general Hilbert space operators. The
second part deals with  Numerical radius inequalities   for $n
\times n$ matrix Operators.

\subsection{ Numerical radius inequalities}
In this section, some numerical radius inequalities based on
results of Section \ref{sec2} are obtained. Before  that, we need
to recall that in some recent works, some authors used the concept
of infimum norm (or $\ell$-norm) which is defined as:
\begin{align*}
\ell \left( T \right): &= \inf \left\{ {\left\| {Tx} \right\|:x
	\in \mathscr{H},\left\| x \right\| = 1} \right\}
\\
&=      \inf \left\{ {\left\langle {Tx,y} \right\rangle :x,y \in
	\mathscr{H},\left\| x \right\| =\left\| y \right\|= 1} \right\}.
\end{align*}

The next result gives a numerical radius bound of product of two
operators based on the refinement of Kittaneh inequality
\eqref{eq3.6}.
\begin{theorem}
	\label{thm2}Let  $T,S\in \mathscr{B}\left( \mathscr{H}\right)$
	such that $|T|S=S^*|T|$. If $f$ and $g$ are nonnegative continuous
	functions on $\left[0,\infty\right)$ satisfying $f(t)g(t) =t$
	$(t\ge0)$, then
	\begin{align}
	\label{eq4.1} w\left(TS\right)   &\le  \frac{1}{2}\left( {\left\|
		S \right\| + \left\| {S^2 } \right\|^{1/2} } \right)  \cdot
	\left[\left\| {f^p \left( {\left| T \right|} \right)} \right\|^2  - \ell   \left( {\left| {\left[ {f^2 \left( {\left| T \right|} \right) - \left\| {f\left( {\left| T \right|} \right)} \right\|^2} \right]} \right|^{p } } \right)\right]^{\frac{1}{2p}} \\
	&\qquad \times\left[\left\| {g^p \left( {\left| T^* \right|}
		\right)} \right\|^2  - \ell  \left( {\left| {\left[ {g^2 \left(
				{\left| T^* \right|} \right) - \left\| {g\left( {\left| T^*
						\right|} \right)} \right\|^2} \right]} \right|^{p} }
	\right)\right]^{\frac{1}{2p}} \nonumber
	\end{align}
	for all $p\ge2$.
\end{theorem}
\begin{proof}
	From the first inequality in \eqref{eq3.6}, we have
	\begin{align*}
	&\left| {\left\langle {TSx,y} \right\rangle } \right|^{2p} \\
	&\le r^{2p}\left( S \right) \left[ {\left\langle { {f^{2p} \left( {\left| T \right|} \right)} x,x} \right\rangle- \left\langle{\left| {{f^2 \left( {\left| T \right|} \right)} - \left\langle{ {f^2 \left( {\left| T \right|} \right)} x,x} \right\rangle 1_{\mathcal{H}} } \right|^px,x} \right\rangle}\right]\\
	&\qquad\times \left[ {\left\langle { {g^{2p} \left( {\left| T^*
					\right|} \right)} y,y} \right\rangle- \left\langle{\left| {{g^2
					\left( {\left| T^* \right|} \right)} - \left\langle{ {g^2 \left(
						{\left| T^* \right|} \right)} y,y} \right\rangle 1_{\mathcal{H}} }
			\right|^py,y} \right\rangle}\right]
	\end{align*}
	Let $y=x$ and taking the supremum over $x \in \mathscr{H}$, we
	observe  that
	\begin{align*}
	&\sup_{\|x\|=1}\left| {\left\langle {TSx,x} \right\rangle } \right|^{2p}  \\&\le r^{2p}\left( S \right)\sup_{\|x\|=1} \left\{  \left[ {\left\langle { {f^{2p} \left( {\left| T \right|} \right)} x,x} \right\rangle- \left\langle{\left| {{f^2 \left( {\left| T \right|} \right)} - \left\langle{ {f^2 \left( {\left| T \right|} \right)} x,x} \right\rangle 1_{\mathcal{H}} } \right|^px,x} \right\rangle}\right]\right.\\
	&\qquad\left. \times \left[ {\left\langle { {g^{2p} \left( {\left|
					T^* \right|} \right)} x,x} \right\rangle- \left\langle{\left|
			{{g^2 \left( {\left| T^* \right|} \right)} - \left\langle{ {g^2
						\left( {\left| T^* \right|} \right)} x,x} \right\rangle
				1_{\mathcal{H}} } \right|^px,x} \right\rangle}\right]\right\}
	\\
	&\le r^{2p}\left( S \right)\sup_{\|x\|=1}  \left[ {\left\langle { {f^{2p} \left( {\left| T \right|} \right)} x,x} \right\rangle- \left\langle{\left| {{f^2 \left( {\left| T \right|} \right)} - \left\langle{ {f^2 \left( {\left| T \right|} \right)} x,x} \right\rangle 1_{\mathcal{H}} } \right|^px,x} \right\rangle}\right]\\
	&\qquad \times \sup_{\|x\|=1}\left[ {\left\langle { {g^{2p} \left(
				{\left| T^* \right|} \right)} x,x} \right\rangle-
		\left\langle{\left| {{g^2 \left( {\left| T^* \right|} \right)} -
				\left\langle{ {g^2 \left( {\left| T^* \right|} \right)} x,x}
				\right\rangle 1_{\mathcal{H}} } \right|^px,x}
		\right\rangle}\right]
\\
	&\le r^{2p}\left( S \right) \left\{  \sup_{\|x\|=1} \left\langle {
		{f^{2p} \left( {\left| T \right|} \right)} x,x}
	\right\rangle\right.
	\\
	&\qquad\qquad\left.- \inf_{\|x\|=1}\left\langle{\left| {{f^2
				\left( {\left| T \right|} \right)} - \sup_{\|x\|=1}\left\langle{
				{f^2 \left( {\left| T \right|} \right)} x,x} \right\rangle
			1_{\mathcal{H}} } \right|^px,x} \right\rangle \right\}
	\\
	&\qquad  \times  \left\{\sup_{\|x\|=1}\left\langle { {g^{2p}
			\left( {\left| T^* \right|} \right)} x,x}
	\right\rangle\right.\\&\qquad\qquad\left.-
	\inf_{\|x\|=1}\left\langle{\left| {{g^2 \left( {\left| T^*
					\right|} \right)} -\sup_{\|x\|=1} \left\langle{ {g^2 \left(
					{\left| T^* \right|} \right)} x,x} \right\rangle 1_{\mathcal{H}} }
		\right|^px,x} \right\rangle \right\}
	\\
	&\le r^{2p}\left( S \right) \cdot
	\left[\left\| {f^p \left( {\left| T \right|} \right)} \right\|^2  - \ell   \left( {\left| {\left[ {f^2 \left( {\left| T \right|} \right) - \left\| {f\left( {\left| T \right|} \right)} \right\|^2} \right]} \right|^{p } } \right)\right]\\
	&\qquad \times\left[\left\| {g^p \left( {\left| T^* \right|}
		\right)} \right\|^2  - \ell  \left( {\left| {\left[ {g^2 \left(
				{\left| T^* \right|} \right) - \left\| {g\left( {\left| T^*
						\right|} \right)} \right\|^2} \right]} \right|^{p} }
	\right)\right].
	\end{align*}
	Now, from Lemma \ref{lemma6} with $A=S$, $B=1_{\mathscr{H}}$, we
	have
	\begin{align*}
	r\left( S \right) \le \frac{1}{4}\left( {2\left\| S \right\| +
		\sqrt {4\min \left\{ {\left\| {S^2 } \right\|,\left\| S \right\|^2
			} \right\}} } \right) = \frac{1}{2}\left( {\left\| S \right\| +
		\left\| {S^2 } \right\|^{1/2} } \right).
	\end{align*}
	Substituting in the above inequality we obtain the result in
	\eqref{eq4.1}.

\end{proof}
\begin{corollary}
	Let  $T,S\in \mathscr{B}\left( \mathscr{H}\right)$ such that
	$|T|S=S^*|T|$. Then
	\begin{align}
	\label{eq4.2} w\left(TS\right)  &\le r\left( S \right) \cdot
	\left[\left\| { \left( {\left| T \right|} \right)^{p\alpha}} \right\|^2  - \ell ^2 \left( {\left| {\left[ { \left( {\left| T \right|} \right)^{2\alpha} - \left\| {\left( {\left| T \right|} \right)^{ \alpha}} \right\|} \right]} \right|^{\frac{p}{2}} } \right)\right]^{\frac{1}{2p}}  \\
	&\qquad \times\left[\left\| {  \left( {\left| T^* \right|}
		\right)^{p\left(1-\alpha\right)}} \right\|^2  - \ell ^2 \left(
	{\left| {\left[ {  \left( {\left| T^* \right|} \right)^{
					2\left(1-\alpha\right)} - \left\| { \left( {\left| T^* \right|}
					\right)^{ \left(1-\alpha\right)}} \right\|} \right]}
		\right|^{\frac{p}{2}} } \right)\right]^{\frac{1}{2p}} \nonumber
	\\
	&\le  r\left( S \right) \cdot \left\| { \left( {\left| T \right|}
		\right)^{p\alpha}} \right\|^{1/p}\left\| {  \left( {\left| T^*
			\right|}
		\right)^{p\left(1-\alpha\right)}} \right\|^{1/p}\nonumber
	\end{align}
	for all $p\ge2$. In particular, we have
	\begin{align}
	\label{eq4.3} w\left(TS\right)  &\le r\left( S \right) \cdot
	\left[\left\| {T} \right\|^2  - \ell ^2 \left( {\left| {\left[ { \left| T \right|- \left\| { \left| T \right| ^{1/2}} \right\|} \right]} \right|  } \right)\right]^{\frac{1}{4}}\\
	&\qquad \times\left[\left\| { T} \right\|^2  - \ell ^2 \left(
	{\left| {\left[ {  \left| T^* \right| - \left\| {  \left| T^*
					\right| ^{ 1/2}} \right\|} \right]} \right|  }
	\right)\right]^{\frac{1}{4}} \nonumber
	\\
	&\le r\left( S \right) \left\|T\right\|.\nonumber
	\end{align}
\end{corollary}
\begin{proof}
	Setting $f\left(t\right)=t^{\alpha}$ and
	$g\left(t\right)=t^{1-\alpha}$ in \eqref{eq4.1},  we get the
	inequality \eqref{eq4.2}. Choosing $p=2$ and $\alpha=\frac{1}{2}$
	in \eqref{eq4.2} and use the fact that $\||T|\|=\||T^*|\|=\|T\|$,
	we get the second inequality \eqref{eq4.3}.
\end{proof}

Another generalization of the above inequalities under Kittaneh's
assumptions is embedded as follows:
\begin{corollary}
	Under the assumptions of Theorem \ref{thm2}, we have
	\begin{align*}
	w\left( {TS}  \right) &\le \frac{1}{4}\left( {\left\| S \right\| +
		\left\| {S^2 } \right\|^{1/2} } \right) \left\|f^{2p} \left(
	{\left| T \right|} \right)+g^{2p} \left( {\left| T^* \right|}
	\right)\right\|
	\\
	&\le \frac{1}{8} \left( {\left\| S \right\| + \left\| {S^2 }
		\right\|^{1/2} } \right) \cdot\left\{\left( {\left\| {f^{p} \left(
			{\left| A \right|} \right)} \right\|^2 + \left\| {g^{p}\left(
			{\left| {A^* } \right|} \right)} \right\|^2} \right) \right.
	\\
	&\qquad\left.+ \sqrt {\left( {\left\| {f^{2p} \left( {\left|A
					\right|} \right)} \right\| - \left\| {g^{2p} \left( {\left| {A^* }
					\right|} \right)} \right\|} \right)^2  + 4\left\|
		{f^{p}\left({\left| A \right|} \right)g^{p}\left( {\left| {A^* }
				\right|} \right)} \right\|}\right\}
	\end{align*}
	for all $p\ge2$.
\end{corollary}
\begin{proof}
	In the second inequality in \eqref{eq3.6}, let $x=y$ then we have
	\begin{align*}
	\left| {\left\langle {TSx,x} \right\rangle } \right|  &\le r\left(
	S \right) \sqrt{\left\langle { {f^{2p} \left( {\left| T \right|}
				\right)} x,x} \right\rangle } \sqrt{\left\langle { {g^{2p} \left(
				{\left| T^* \right|} \right)} x,x} \right\rangle }
	\\
	&\le \frac{1}{2} r\left( S \right) \left( {\left\langle { {f^{2p}
				\left( {\left| T \right|} \right)} x,x} \right\rangle
		+\left\langle { {g^{2p} \left( {\left| T^* \right|} \right)} x,x}
		\right\rangle }\right) \,\,\,\,{\rm{(by\, AM\text{-}GM \,
			inequality)}}
	\\
	&\le \frac{1}{2} r\left( S \right) \left\|f^{2p} \left( {\left| T
		\right|} \right)+g^{2p} \left( {\left| T^* \right|}
	\right)\right\|
	\\
	&\le \frac{1}{4}r\left( S \right)\left( {\left\| {f^{2p} \left(
			{\left| A \right|} \right)} \right\| + \left\| {g^{2p}\left(
			{\left| {A^* } \right|} \right)} \right\|} \right)
	\\
	&\qquad+ \frac{1}{4}\sqrt {\left( {\left\| {f^{2p} \left( {\left|
					A \right|} \right)} \right\| - \left\| {g^{2p} \left( {\left| {A^*
					} \right|} \right)} \right\|} \right)^2  + 4\left\|
		{f^{\frac{p}{2}}\left( {\left| A \right|}
			\right)g^{\frac{p}{2}}\left( {\left| {A^* } \right|} \right)}
		\right\|^2}
	\end{align*}
	Now, using  \eqref{fact1} and    \eqref{fact2} in the last
	inequality  and use the inequality
	\begin{align*}
	r\left( S \right) \le   \frac{1}{2}\left( {\left\| S \right\| +
		\left\| {S^2 } \right\|^{1/2} } \right),
	\end{align*}
	Substituting all together   in the last  inequality and taking the
	supremum for all $x\in \mathscr{H}$, we get the desired result.
\end{proof}

\subsection{ Numerical radius inequalities   for $n \times n$ matrix Operators}
On the other hand, several refinements inequalities for  numerical
radius of $n\times n$ operator matrices have been recently
obtained by many other authors see for example \cite{AF2},
\cite{D1}--\cite{D3}, \cite{FK1}--\cite{FK3}, \cite{OMN}. Among
others, three important facts are obtained by different authors
are summarized together in the following result.

Let  $A=\left[A_{ij}\right]\in \mathscr{B}\left(\bigoplus _{i =
	1}^n \mathscr{H}_i\right)$. Then
\begin{align}
\label{eq1.6} w\left(A\right)\le \left\{ \begin{array}{l}
\omega \left( {\left[ {t_{ij}^{\left( 1 \right)} } \right]} \right),\qquad {\rm{Hou \,\&\, Du \,\,in}\,\,}\text{\cite{HD}} \\
\\
\omega \left( {\left[ {t_{ij}^{\left( 2 \right)} } \right]} \right) ,\qquad {\rm{BaniDomi \,\&\, Kittaneh \,\,in}\,\,} \text{\cite{BF}}\\
\\
\omega \left( {\left[ {t_{ij}^{\left( 3 \right)} } \right]} \right),\qquad  {\rm{AbuOmar \,\&\, Kittaneh \,\,in}\,\,} \text{\cite{AF1}}\\
\end{array} \right.;
\end{align}
where
\begin{align*}
t_{ij}^{\left( 1 \right)}  &= \omega \left( {\left[ {\left\|
		{T_{ij} } \right\|} \right]} \right) ,
\qquad
t_{ij}^{\left( 2 \right)}  = \left\{ \begin{array}{l}
\frac{1}{2}\left( {\left\| {T_{ii} } \right\| + \left\| {T_{ii}^2 } \right\|^{1/2} } \right),\,\,\,\,\,\,\,i = j \\
\\
\left\| {T_{ij} } \right\|,\qquad\qquad\qquad\,\,\,\,\,\,\,\,\,\,i \ne j
\end{array} \right. ,
\end{align*}
and
\begin{align*}
t_{ij}^{\left( 3 \right)}  = \left\{ \begin{array}{l}
\omega \left( {T_{ii} } \right),\,\,\,\,\,\,\,i = j \\
\\
\left\| {T_{ij} } \right\|,\,\,\,\,\,\,\,\,\,\, i \ne j \\
\end{array} \right.
\end{align*}

Our next result gives a new bound for Numerical radius of $n
\times n$ matrix Hilbert Operators.
\begin{theorem}
	Let  $A=\left[A_{ij}\right]\in \mathscr{B}\left(\bigoplus _{i =
		1}^n \mathscr{H}_i\right)$ and $f,g$ be as in Lemma \ref{lemma5}.
	Then
	\begin{align}
	\label{eq4.4}w\left(A\right)\le w\left(\left[a_{ij}\right]\right)
	\end{align}
	where
	\begin{align*}
	a_{ij}  = \left\{ \begin{array}{l}
	\frac{1}{4}B_{ii},\qquad i = j \\
	\\
	\left\| {A_{ij} } \right\|,\,\,\,\,\,\,\,\, i \ne j \\
	\end{array} \right.
	\end{align*}
	such that
	\begin{align*}
	B_{ii}&=\left\| {f^2 \left( {\left| {A_{ii} } \right|} \right)}
	\right\| + \left\| {g^2 \left( {\left| {A_{ii}^* } \right|}
		\right)} \right\|
	\\
	&\qquad+ \sqrt {\left( {\left\| {f^2 \left( {\left| {A_{ii} }
					\right|} \right)} \right\| - \left\| {g^2 \left( {\left| {A_{ii}^*
					} \right|} \right)} \right\|} \right)^2  +  4\left\| {f \left(
			{\left| {A_{ii} } \right|} \right)g \left( {\left| {A_{ii}^* }
				\right|} \right)} \right\|^{2} }
	\end{align*}
	
\end{theorem}

\begin{proof}
	Let $x=\left( {\begin{array}{*{20}c}{ x_1  } & { x_2 } &  \cdots &
		{  x_n }  \\    \end{array}} \right)^T\in \bigoplus _{i = 1}^n
	\mathscr{H}_i$, with $\|x\|=1$. Then we have
	\begin{align}
	& \left| {\left\langle {Ax,x} \right\rangle } \right|\nonumber \\&= \left| {\sum\limits_{i,j = 1}^n {\left\langle {A_{ij} x_j ,x_i } \right\rangle } } \right|\nonumber \\
	&\le \sum\limits_{i,j = 1}^n {\left| {\left\langle {A_{ij} x_j ,x_i } \right\rangle } \right|}  \nonumber\\
	&= \sum\limits_{i = 1}^n {\left| {\left\langle {A_{ii} x_i ,x_i } \right\rangle } \right|}  + \sum\limits_{j \ne i}^n {\left| {\left\langle {A_{ij} x_j ,x_i } \right\rangle } \right|}  \nonumber
	\\
	&\le \sum\limits_{i = 1}^n {\left\langle {f^2 \left( {\left| {A_{ii} } \right|} \right)x_i ,x_i } \right\rangle ^{1/2} \left\langle {g^2 \left( {\left| {A_{ii}^* } \right|} \right)x_i ,x_i } \right\rangle ^{1/2} } +  \sum\limits_{j \ne i}^n {\left| {\left\langle {A_{ij} x_j ,x_i } \right\rangle } \right|}\qquad ({\rm{by}\,\,\eqref{kittaneh.ineq}})  \label{eq4.5}
	\\
	&\le \frac{1}{2}\sum\limits_{i = 1}^n {\left\| {f^2 \left( {\left| {A_{ii} } \right|} \right) + g^2 \left( {\left| {A_{ii}^* } \right|} \right)} \right\|\left\| {x_i } \right\|^2 }  + \sum\limits_{j \ne i}^n {\left\| {A_{ij} } \right\|\left\| {x_i } \right\|\left\| {x_j } \right\|} \nonumber \,\,\,\,{\rm{(by\, AM\text{-}GM \, inequality)}}
	\end{align}
		\begin{align*}
	&\le \frac{1}{4}  \sum\limits_{i = 1}^n {\left( {\left\| {f^2 \left( {\left| {A_{ii} } \right|} \right)} \right\| + \left\| {g^2 \left( {\left| {A_{ii}^* } \right|} \right)} \right\|} \right.} \,\,\,\,\,\,\qquad\qquad({\rm{by}}\,\,\eqref{fact1}) \nonumber
	\\
	&\qquad\left. { + \sqrt {\left( {\left\| {f^2 \left( {\left| {A_{ii} } \right|} \right)} \right\| - \left\| {g^2 \left( {\left| {A_{ii}^* } \right|} \right)} \right\|} \right)^2  + 4\left\| {f \left( {\left| {A_{ii} } \right|} \right)g \left( {\left| {A_{ii}^* } \right|} \right)} \right\|^{2} } } \right)\left\| {x_i } \right\|^2  \nonumber
	\\
	&\qquad+ \sum\limits_{j \ne i}^n {\left\| {A_{ij} } \right\|\left\| {x_i } \right\|\left\| {x_j } \right\|} \nonumber
	\\
	& = \left\langle {\left[ {a_{ij} } \right]y,y} \right\rangle\nonumber
	\end{align*}
	where $y=\left( {\begin{array}{*{20}c}{\left\| {x_1 } \right\|} &
		{\left\| {x_2 } \right\|} &  \cdots  & {\left\| {x_n } \right\|}
		\\  \end{array}} \right)^T$. Taking the supremum for all $x\in \mathscr{H}$, we get the desired result.
\end{proof}

\begin{corollary}
	If $\bf{A}=\left[ {\begin{array}{*{20}c}
		{A_{11} } & {A_{12} }  \\
		{A_{21} } & {A_{22} }  \\
		\end{array}} \right]$ in $ \mathscr{B}\left(\mathscr{H}_1\oplus\mathscr{H}_2\right)$ , then
	\begin{align}
	\label{eq4.6}w\left(A\right)\le
	w\left(\left[\widehat{a}_{ij}\right]\right)
	\end{align}
	where
	\begin{align*}
	\widehat{a}_{ij}  = \left\{ \begin{array}{l}
	\frac{1}{4}\widehat{B}_{ii},\qquad i = j \\
	\\
	\left\| {A_{ij} } \right\|,\,\,\,\,\,\,\,\,\, i \ne j \\
	\end{array} \right.
	\end{align*}
	such that
	\begin{align*}
	\widehat{B}_{ii}&=\left\| {\left| {A_{ii} } \right|^{2\alpha } } \right\| + \left\| {\left| {A_{ii}^* } \right|^{2\left( {1 - \alpha } \right)} } \right\| \\&\qquad+ \sqrt {\left( {\left\| {\left| {A_{ii} } \right|^{2\alpha } } \right\| - \left\| {\left| {A_{ii}^* } \right|^{2\left( {1 - \alpha } \right)} } \right\|} \right)^2  + 4\left\| {\left| {A_{ii} } \right|^\alpha  \left| {A_{ii}^* } \right|^{\left( {1 - \alpha } \right)} } \right\|^2 }\\
	&\left(:=\widehat{B}_{ii}\left(\alpha\right) \right)
	\end{align*}
\end{corollary}

\begin{proof}
	Setting $f\left(t\right)=t^\alpha$ and $g\left(t\right)=t^{1-\alpha}$ in \eqref{eq4.4}, then we get
	\begin{align*}
	w\left( \left[ {\begin{array}{*{20}c}
		{A_{11} } & {A_{12} }  \\
		{A_{21} } & {A_{22} }  \\
		\end{array}} \right]\right) &\le w
	\left( {\begin{array}{*{20}c}
		\begin{array}{l}
		\frac{1}{4}\widehat{B}_{11}  \\
		\\
		\end{array} & \begin{array}{l}
		\left\| {A_{12} } \right\|\\
		\\
		\end{array}  \\
		\left\| {A_{21} } \right\| & \frac{1}{4}\widehat{B}_{22} \\
		\end{array}} \right)
	\\
	&= \frac{1}{2}r         \left( {\begin{array}{*{20}c}
		\begin{array}{l}
		\frac{1}{2}\widehat{B}_{11}  \\
		\\
		\end{array} & \begin{array}{l}
		\left\| {A_{12} } \right\|+\left\| {A_{21} } \right\|\\
		\\
		\end{array}  \\
		\left\| {A_{21} } \right\|+\left\| {A_{12} } \right\| & \frac{1}{2}\widehat{B}_{22} \\
		\end{array}} \right)
	\\
	&=  \frac{1}{4}\left(\widehat{B}_{11}+ \widehat{B}_{22}+ \sqrt {\left( { \widehat{B}_{11}-\widehat{B}_{22}} \right)^2  +  \left(    \left\| {A_{12} } \right\|+\left\| {A_{21} } \right\| \right)^2 }  \right)
	\end{align*}
	which gives the required result.
\end{proof}

\begin{remark}
	Setting $\alpha=\frac{1}{2}$ in \eqref{eq4.6}  and employing the
	facts \eqref{fact1} and \eqref{fact2}, so that we get
	\eqref{eq1.2}.
\end{remark}

\begin{theorem}
	\label{thm4}    Let  $A=\left[A_{ij}\right]\in
	\mathscr{B}\left(\bigoplus _{i = 1}^n \mathscr{H}_i\right)$ and
	$f,g$ be as in Lemma \ref{lemma3}. Then
	\begin{align}
	\label{eq4.7}w\left(A\right)\le w\left(\left[h_{ij}\right]\right)
	\end{align}
	where
	\begin{align*}
	h_{ij}  = \left\{ \begin{array}{l}
	\frac{1}{4}\left( {D_{ii}  - d_{ii} } \right),\,\,\,\,\,i = j \\
	\\
	\left\| {A_{ij} } \right\|,\qquad\,\,\,\,\,\,\,\,\,\,i \ne j \\
	\end{array} \right.
	\end{align*}
	such that
	\begin{align*}
	D_{ii} & = \frac{1}{2}\left( {\left\| {f^4 \left( {\left| {A_{ii}
				} \right|} \right)} \right\| + \left\| {g^4 \left( {\left|
				{A_{ii}^* } \right|} \right)} \right\| }\right.\\&\qquad\left.{+
		\sqrt {\left( {\left\| {f^4 \left( {\left| {A_{ii} } \right|}
					\right)} \right\| - \left\| {g^4 \left( {\left| {A_{ii}^* }
						\right|} \right)} \right\|} \right)^2  + 4\left\| {f^2 \left(
				{\left| {A_{ii} } \right|} \right)g^2 \left( {\left| {A_{ii}^* }
					\right|} \right)} \right\|^{1/2} } } \right)
	\end{align*}
	and
	\begin{align*}
	d_{ii}  &= \left\| {\left| {f^2\left( {\left| {A_{ii} } \right|}
			\right) - \left\| {f^2\left( {\left| {A_{ii} } \right|} \right)}
			\right\|} \right|^2 + \left| {g^2\left( {\left| {A_{ii}^* }
				\right|} \right) - \left\| {g^2\left( {\left| {A_{ii}^* } \right|}
				\right)} \right\|} \right|^2} \right\|
	\end{align*}
\end{theorem}
\begin{proof}
	From \eqref{eq4.5} we have
	\begin{align*}
	&\left| {\left\langle {Ax,x} \right\rangle } \right| \\
	&\le \sum\limits_{i = 1}^n {\left\langle {f^2 \left( {\left| {A_{ii} } \right|} \right)x_i ,x_i } \right\rangle ^{1/2} \left\langle {g^2 \left( {\left| {A_{ii}^* } \right|} \right)x_i ,x_i } \right\rangle ^{1/2} }  + \sum\limits_{j \ne i}^n {\left| {\left\langle {A_{ij} x_j ,x_i } \right\rangle } \right|}
	\\
	&\le \sum_{i=1}^n \left\{\sqrt[{4}]{\left\langle { {f^{4} \left( {\left| A_{ii} \right|} \right)} x,x} \right\rangle- \left\langle{\left| {{f^2 \left( {\left|  A_{ii} \right|} \right)} - \left\langle{ {f^2 \left( {\left|  A_{ii} \right|} \right)} x,x} \right\rangle 1_{\mathcal{H}} } \right|^2x,x} \right\rangle}\right.
	\\
	&\qquad\left.\times \sqrt[{4}]{\left\langle { {g^{4} \left( {\left|  A_{ii}^* \right|} \right)} y,y} \right\rangle- \left\langle{\left| {{g^2 \left( {\left|  A_{ii}^* \right|} \right)} - \left\langle { {g^2 \left( {\left|  A_{ii}^* \right|} \right)} y,y} \right\rangle 1_{\mathcal{H}} } \right|^2y,y} \right\rangle} \cdot \left\| {x_i } \right\|^2\right\}
	\\
	&\qquad+ \sum\limits_{j \ne i}^n {\left\| {A_{ij} } \right\|\left\| {x_i } \right\|\left\| {x_j } \right\|} \qquad \qquad\qquad({\rm{by\,\, \eqref{eq3.8}\,\, with}}\, S=1_{\mathscr{H}}) \nonumber
	\\
	&\le  \frac{1}{4} \sum\limits_{i = 1}^n \left \{ {\left\| {f^4 \left( {\left| {A_{ii} } \right|} \right) + g^4 \left( {\left| {A_{ii}^* } \right|} \right)} \right\| } \right\}
	\\
	&\qquad\left.   - \left\| {\left| {f^2\left( {\left| {A_{ii} }
				\right|} \right) - \left\| {f^2\left( {\left| {A_{ii} } \right|}
				\right)} \right\|} \right|^2 + \left| {g^2\left( {\left| {A_{ii}^*
				} \right|} \right) - \left\| {g^2\left( {\left| {A_{ii}^* }
					\right|} \right)} \right\|} \right|^2} \right\| \right\}\left\|
	{x_i } \right\|^2
	\nonumber\\
	&\qquad+ \sum\limits_{j \ne i}^n {\left\| {A_{ij} } \right\|\left\| {x_i } \right\|\left\| {x_j } \right\|} \qquad\qquad\qquad ({\rm{by\,GM-AM\, inequality}}) \nonumber
	\\
	&\le  \sum\limits_{i = 1}^n \left\{ \frac{1}{8}\left( {\left\| {f^4 \left( {\left| {A_{ii} } \right|} \right)} \right\| + \left\| {g^4 \left( {\left| {A_{ii}^* } \right|} \right)} \right\|} \right. \right.
	\\
	&\qquad \left. { + \sqrt {\left( {\left\| {f^4 \left( {\left| {A_{ii} } \right|} \right)} \right\| - \left\| {g^4 \left( {\left| {A_{ii}^* } \right|} \right)} \right\|} \right)^2  + 4\left\| {f^2 \left( {\left| {A_{ii} } \right|} \right)g^2 \left( {\left| {A_{ii}^* } \right|} \right)} \right\|^{2} } } \right)
	\\
	&\qquad \left.{ - \frac{1}{4}\left\| {\left| {f^2\left( {\left| {A_{ii} } \right|} \right) - \left\| {f^2\left( {\left| {A_{ii} } \right|} \right)} \right\|} \right|^2 + \left| {g^2\left( {\left| {A_{ii}^* } \right|} \right) - \left\| {g^2\left( {\left| {A_{ii}^* } \right|} \right)} \right\|} \right|^2} \right\|} \right\}\cdot\left\| {x_i } \right\|^2
	\nonumber\\
	&\qquad+ \sum\limits_{j \ne i}^n {\left\| {A_{ij} } \right\|\left\| {x_i } \right\|\left\| {x_j } \right\|} \nonumber
	\end{align*}
\begin{align*}
	&=\frac{1}{8} \sum\limits_{i = 1}^n     \left( {\left\| {f^4 \left( {\left| {A_{ii} } \right|} \right)} \right\| + \left\| {g^4 \left( {\left| {A_{ii}^* } \right|} \right)} \right\|} \right.
	\\
	&\qquad \qquad\left. { + \sqrt {\left( {\left\| {f^4 \left( {\left| {A_{ii} } \right|} \right)} \right\| - \left\| {g^4 \left( {\left| {A_{ii}^* } \right|} \right)} \right\|} \right)^2  + 4\left\| {f^2 \left( {\left| {A_{ii} } \right|} \right)g^2 \left( {\left| {A_{ii}^* } \right|} \right)} \right\|^{2} } } \right)\cdot\left\| {x_i } \right\|^2
	\\
	&\qquad  - \sum\limits_{i = 1}^n \frac{1}{4}\left\| {\left| {f^2\left( {\left| {A_{ii} } \right|} \right) - \left\| {f^2\left( {\left| {A_{ii} } \right|} \right)} \right\|} \right|^2 + \left| {g^2\left( {\left| {A_{ii}^* } \right|} \right) - \left\| {g^2\left( {\left| {A_{ii}^* } \right|} \right)} \right\|} \right|^2} \right\| \cdot\left\| {x_i } \right\|^2
	\nonumber
	\\
	&\qquad+ \sum\limits_{j \ne i}^n {\left\| {A_{ij} } \right\|\left\| {x_i } \right\|\left\| {x_j } \right\|} \nonumber
	\\
	&= \left\langle {\left[ {h_{ij} } \right]y,y} \right\rangle\nonumber
	\end{align*}
	where $y=\left( {\begin{array}{*{20}c}{\left\| {x_1 } \right\|} & {\left\| {x_2 } \right\|} &  \cdots  & {\left\| {x_n } \right\|}  \\  \end{array}} \right)^T$. Taking the supremum for all $x\in \mathscr{H}$, we get the desired result.
\end{proof}

\begin{corollary}
	\label{cor7} If $\bf{A}=\left[ {\begin{array}{*{20}c}
		{A_{11} } & {A_{12} }  \\
		{A_{21} } & {A_{22} }  \\
		\end{array}} \right]$ in $ \mathscr{B}\left(\mathscr{H}_1\oplus\mathscr{H}_2\right)$ , then
	\begin{align*}
	&w\left(    \left[ {\begin{array}{*{20}c}
		{A_{11} } & {A_{12} }  \\
		{A_{21} } & {A_{22} }  \\
		\end{array}} \right]\right)  \\ &\le \frac{1}{4}\left\{ \left( {\widetilde{D}_{11} - \widetilde{d}_{11}  } \right)+\left( {\widetilde{D}_{22} - \widetilde{d}_{22}  } \right)\right.\\
	&\qquad \left.+ \sqrt {\left( { \left( {\widetilde{D}_{11} -
				\widetilde{d}_{11}  } \right)-\left( {\widetilde{D}_{22} -
				\widetilde{d}_{22}  } \right)} \right)^2  +  \left( \left\|
		{A_{12} } \right\|+\left\| {A_{21} } \right\| \right)^2 } \right\}
	\end{align*}
	where
	\begin{align*}
	\widetilde{h}_{ij}  = \left\{ \begin{array}{l}
	\frac{1}{4}\left( {\widetilde{D}_{ii} - \widetilde{d}_{ii}  } \right),\,\,\,\,\,i = j \\
	\\
	\left\| {A_{ij} } \right\|,\qquad\,\,\,\,\,\,\,\,\,\,i \ne j \\
	\end{array} \right.
	\end{align*}
	such that
	\begin{align*}
	\widetilde{D}_{ii}  &= \frac{1}{2}\left( {\left\| {\left| {A_{ii}
			} \right|^{4\alpha } } \right\| + \left\| {\left| {A_{ii}^* }
			\right|^{4\left( {1 - \alpha } \right)} } \right\| }\right.\\
	&\qquad\left.{+ \sqrt {\left( {\left\| {\left| {A_{ii} }
					\right|^{4\alpha } } \right\| - \left\| {\left| {A_{ii}^* }
					\right|^{4\left( {1 - \alpha } \right)} } \right\|} \right)^2  +
			4\left\| {\left| {A_{ii} } \right|^{2\alpha}  \left| {A_{ii}^* }
				\right|^{2\left( {1 - \alpha } \right)} } \right\|^{2} } } \right)
	\end{align*}
	and
	\begin{align*}
	\widetilde{d}_{ii}  = \left\| {\left| {\left| {A_{ii} }
			\right|^{2\alpha}   - \left\| {\left| {A_{ii} } \right|^{2\alpha}
			} \right\|} \right|^2 + \left| {\left| {A_{ii}^* }
			\right|^{2\left(1 - \alpha\right)}  - \left\| {\left| {A_{ii}^* }
				\right|^{2\left(1 - \alpha\right) } } \right\|} \right|^2}
	\right\|
	\end{align*}
\end{corollary}

\begin{proof}
	Setting $f\left(t\right)=t^\alpha$ and $g\left(t\right)=t^{1-\alpha}$ in \eqref{eq4.7}, then we get
	\begin{align*}
	&w\left(    \left[ {\begin{array}{*{20}c}
		{A_{11} } & {A_{12} }  \\
		{A_{21} } & {A_{22} }  \\
		\end{array}} \right]\right) \\&\le w
	\left(  \left[ {\begin{array}{*{20}c}
		{\frac{1}{4}\left( {\widetilde{D}_{11} - \widetilde{d}_{11}  } \right) } & {\left\| {A_{12} } \right\| }  \\
		{\left\| {A_{21} } \right\|} & {\frac{1}{4}\left( {\widetilde{D}_{22} - \widetilde{d}_{22}  } \right)}  \\
		\end{array}} \right]\right)
	\\
	&= \frac{1}{2}r     \left(  \left[ {\begin{array}{*{20}c}
		{\frac{1}{2}\left( {\widetilde{D}_{11} - \widetilde{d}_{11}  } \right) } & {\left\| {A_{12} } \right\| +\left\| {A_{21} } \right\|}  \\
		{\left\| {A_{21} } \right\|+\left\| {A_{12} } \right\|} & {\frac{1}{2}\left( {\widetilde{D}_{22} - \widetilde{d}_{22}  } \right)}  \\
		\end{array}} \right]\right)
	\\
	&=  \frac{1}{4}\left\{ \left( {\widetilde{D}_{11} - \widetilde{d}_{11}  } \right)+\left( {\widetilde{D}_{22} - \widetilde{d}_{22}  } \right)\right.\\
	&\qquad\left.+ \sqrt {\left( { \left( {\widetilde{D}_{11} - \widetilde{d}_{11}  } \right)-\left( {\widetilde{D}_{22} - \widetilde{d}_{22}  } \right)} \right)^2  +  \left(  \left\| {A_{12} } \right\|+\left\| {A_{21} } \right\| \right)^2 }   \right\}.
	\end{align*}
\end{proof}

The following results refines the first and the second
inequalities in \eqref{eq1.6}
\begin{corollary}
	\label{cor8}If $\bf{A}=\left[ {\begin{array}{*{20}c}
		{A_{11} } & {A_{12} }  \\
		{A_{21} } & {A_{22} }  \\
		\end{array}} \right]$ in $ \mathscr{B}\left(\mathscr{H}_1\oplus\mathscr{H}_2\right)$ , then
	\begin{align*}
	w\left( \left[ {\begin{array}{*{20}c}
		{A_{11} } & {A_{12} }  \\
		{A_{21} } & {A_{22} }  \\
		\end{array}} \right]\right)   \le \frac{1}{4}\left(  \widetilde{R}_{11} + \widetilde{R}_{22}    + \sqrt { \left( {\widetilde{R}_{11} - \widetilde{R}_{22}  } \right)^2   +  \left(    \left\| {A_{12} } \right\|+\left\| {A_{21} } \right\| \right)^2 }   \right),
	\end{align*}
	
	where
	\begin{align*}
	\widetilde{h}_{ij}  = \left\{ \begin{array}{l}
	R_{ii},\qquad i = j \\
	\\
	\left\| {A_{ij} } \right\|,\qquad i \ne j \\
	\end{array} \right.
	\end{align*}
	such that $R_{ii}=\frac{1}{2}\left\| {A^2_{ii} } \right\|-
	\frac{1}{4} \left\| {\left| {\left| {A_{ii} } \right|   - \left\|
			{ A_{ii}    } \right\|} \right|^2 + \left| {\left| {A_{ii}^* }
			\right|  - \left\| { A_{ii}    } \right\|} \right|^2} \right\|$
\end{corollary}
\begin{proof}
	Setting $\alpha=\frac{1}{2}$ in Corollary \ref{cor7}.
\end{proof}
Clearly, the obtained bounds in  Corollary \ref{cor8} are better
than  the first and the second bounds in \eqref{eq1.6}.

\centerline{}

\centerline{}

\end{document}